\documentclass[12pt,a4paper]{amsart}
\usepackage{amssymb}
\usepackage{color}

\input{prepictex}
\input{pictex}
\input{postpictex}

\definecolor{MyBlue}{rgb}{.15,.1,.9}
\definecolor{MyRed}{rgb}{.9,.1,.05}
\definecolor{MyGreen}{rgb}{.1,.7,.3}

\newcommand{\clsp}{\overline{\operatorname{span}}}
\newcommand{\opsp}{\operatorname{span}}
\newcommand{\Aut}{\operatorname{Aut}}

\newcommand{\Obj}{\operatorname{Obj}}
\newcommand{\Mor}{\operatorname{Mor}}
\newcommand{\dom}{\operatorname{dom}}
\newcommand{\cod}{\operatorname{cod}}
\newcommand{\id}{\operatorname{id}}
\newcommand{\ZZ}{\mathbb{Z}}
\newcommand{\CC}{\mathbb{C}}
\newcommand{\DD}{\mathbb{D}}
\newcommand{\NN}{\mathbb{N}}
\newcommand{\TT}{\mathbb{T}}

\newcommand{\Oo}{\mathcal{O}}
\newcommand{\Kk}{\mathcal{K}}
\newcommand{\Cc}{\mathcal{C}}
\newcommand{\Ll}{\mathcal{L}}
\newcommand{\Aa}{\mathcal{A}}
\newcommand{\Bb}{\mathcal{B}}

\newcommand{\dsc}{\oplus_C}
\newcommand{\dsz}{\oplus_Z}

\newcommand{\hf}{H_\infty^{\rm fin}}

\newtheorem{theorem}{Theorem}[section]

\newtheorem{lemma}[theorem]{Lemma}
\newtheorem*{lemma*}{Lemma}
\newtheorem{prop}[theorem]{Proposition}

\theoremstyle{remark}

\newtheorem{rmk}[theorem]{Remark}
\newtheorem{exm}[theorem]{Example}
\newtheorem{exm*}{Example}

\theoremstyle{definition}
\newtheorem{defn}[theorem]{Definition}

\begin{document}

\title[$C^*$-algebras associated to $C^*$-correspondences]{$C^*$-algebras associated to
$C^*$-correspondences and applications to mirror quantum spheres}

\author{David Robertson}
\author{Wojciech Szyma{\'n}ski}

\date{10 January 2010}

\renewcommand{\sectionmark}[1]{}

\vspace{7mm}
\begin{abstract}
The structure of the $C^*$-algebras corresponding to even-dimensional mirror quantum
spheres is investigated. It is shown that they are isomorphic to both Cuntz-Pimsner
algebras of certain $C^*$-correspondences and $C^*$-algebras of certain labelled graphs.
In order to achieve this, categories of labelled graphs and
$C^*$-correspondences are studied. A functor from labelled graphs to $C^*$-correspondences
is constructed, such that the corresponding associated $C^*$-algebras are isomorphic.
Furthermore, it is shown that $C^*$-correspondences for the mirror quantum spheres arise
via a general construction of restricted direct sum.
\end{abstract}

\maketitle

\vfill

\vspace{3mm}

\vspace{3mm}
\noindent This work has been supported in part by: the FNU Rammebevilling grant `Operator
algebras and applications' (2009--2011), the Marie Curie Research Training Network
MRTN-CT-2006-031962 EU-NCG, the NordForsk Research Network `Operator algebra and dynamics',
the Marie Curie International Staff Exchange Scheme PIRSES-GA-2008-230836,
and the Polish Government grant N201 1770 33.

\newpage

\section{Introduction}

Our original motivation for this study came from the desire to better understand the
$C^*$-algebraic structure of mirror quantum spheres. They were first defined in dimension
2 in \cite{HaMaSz}, and then in full generality in \cite{HoSz3,HoSz}. It was noted from the
beginning that in even dimensions these differ on the $C^*$-algebraic level from the
Euclidean quantum spheres arising from quantum groups. In particular, the Euclidean
quantum spheres correspond to certain graph $C^*$-algebras, \cite{HaLa,HoSz2}, but such
a convenient description was lacking for the even dimensional mirror quantum spheres.

In the present paper, we show that the $C^*$-algebras of even dimensional mirror
quantum spheres are naturally isomorphic to both Cuntz-Pimsner algebras of certain
$C^*$-correspondences (in the sense of \cite{Ka1}) and $C^*$-algebras of certain
labelled graphs (in the sense of \cite{BaPa,BaPa2}). We arrive at these two realisations
in the following way. At first, we define a category of $C^*$-correspondences and show that
the process of associating $C^*$-algebras to them is functorial. Furthermore, this
functor sends restricted direct sums of $C^*$-correspondences to pull-backs of their
$C^*$-algebras. Restricted direct sums of Hilbert modules were introduced in \cite{BaGu2},
and this construction is extended to $C^*$-correspondences in the present article.
We then use a realisation of the mirror quantum spheres as pull-backs of quantum discs
to obtain the appropriate $C^*$-correspondences. To this end, we must first find
suitable $C^*$-correspondences for the quantum discs.

A close examination of the above mentioned $C^*$-correspondences results in finding
labelled graphs for the mirror quantum spheres. The explicit forms of the isomorphisms are
obtained by comparing the defining relations.

Our paper is organised as follows.
We begin Section \ref{sec:correspondences} by recalling the definition of $C^*$-correspondences
and the $C^*$-algebra associated to them, following the approach of Katsura \cite{Ka1}. We also
show that the class of $C^*$-correspondences, together with appropriately defined morphisms forms
a category. Thus we obtain a functor which maps a $C^*$-correspondence $(X,A)$ to the
$C^*$-algebra $\Oo_X$, and sends morphisms between $C^*$-correspondences to $C^*$-homomorphisms.

In Section \ref{sec:gluing}, we consider restricted direct sums of $C^*$-correspondences,
which were first defined on the level of Hilbert modules by Baki\'c and Gulja\v{s} in \cite{BaGu2}.
We show that this construction lifts to a pull-back on the level of corresponding $C^*$-algebras.

We use this result in Section \ref{sec:mirrorsphere} in order to obtain a new representation
of the even dimensional mirror quantum spheres. The $C^*$-algebra of such a quantum sphere is
defined as the pull-back of two copies of the
$2n$ dimensional quantum disc algebra $C(\DD^{2n}_q)$ over the obvious surjection to
the $2n-1$ dimension quantum sphere algebra $C(S^{2n-1}_q)$, with one copy of the
disc pre-composed with a `flip' automorphism. We show that they can be realised as
$C^*$-algebras associated to certain explicitly constructed $C^*$-correspondences.

In Section \ref{sec:labelledgraphs}, we first introduce a category of labelled graphs and
construct faithful representations of $C^*$-algebras associated to them. Then we show existence
of a functor from this category of labelled graphs to the category of $C^*$-correspondences
which preserves the associated $C^*$-algebras. Finally, we explicitly construct labelled
graphs giving rise to the $C^*$-algebras of the even dimensional mirror quantum spheres.

\section{Category of $C^*$-correspondences} \label{sec:correspondences}

In this section, we construct a category of $C^*$-correspondences in such a way that the
association of $C^*$-algebras to the correspondences becomes functorial.
These $C^*$-algebras first arose in the work of Pimsner in \cite{Pi}. However, in this original
definition the left action was required to be injective as non-injective left actions often
led to degenerate $C^*$-algebras. This definition included a large class of $C^*$-algebras,
for example crossed product $C^*$-algebras and Cuntz-Krieger algebras. However,
the above mentioned restriction did lead to the exclusion of many interesting $C^*$-algebras,
notably graph algebras where the underlying directed graph had sinks. This problem was
resolved by Katsura in \cite{Ka1}, and his definition reduces to that of Pimsner when
the left action is injective. This larger class of $C^*$-algebras now contains examples such as
crossed products by partial automorphisms, graph algebras associated to graphs with sinks, and
as we shall show in this paper, the labelled graph algebras of Bates and Pask, \cite{BaPa}.

We begin this section by recalling the definition of a $C^*$-correspondence,
e.g. see \cite{La,Ka1}.

For a $C^*$-algebra $A$, a \emph{right Hilbert $A$-module} is a Banach space $X$ equipped
with a right action of $A$, and an $A$-valued inner-product satisfying
\begin{enumerate}
 \item $\langle \xi,\eta a \rangle = \langle \xi,\eta\rangle  a$;
 \item $\langle \eta,\xi\rangle  = \langle \xi,\eta\rangle ^*$; and
 \item $\langle \xi,\xi\rangle  \geq 0$ and $\|\xi\| = \sqrt{\|\langle \xi,\xi\rangle \|}$
\end{enumerate}
for all $\xi, \eta \in X$ and $a \in A$.

We denote by $\Ll(X)$ the set of all \emph{adjointable} operators on $X$; that is,
linear operators $T:X\to X$ such that there exists a linear operator $T^*$ called
the \emph{adjoint} of $T$ such that
\[
 \langle T \xi,\eta\rangle  = \langle \xi, T^* \eta\rangle
\]
for all $\xi, \eta\ \in X$. If the adjoint $T^*$ exists, it is unique.
With the usual operator norm $\|T\|=\sup\{\|Tx\|:\|x\|\leq 1\}$, $\Ll(X)$ is a $C^*$-algebra

For $\xi,\eta\in X$, define $\theta_{\xi,\eta}$ to be the operator satisfying
\[
 \theta_{\xi,\eta}(\zeta) = \xi\langle \eta,\zeta\rangle .
\]

This is an adjointable operator with $(\theta_{\xi,\eta})^* = \theta_{\eta,\xi}$. We call
\[
 \Kk(X) = \clsp\{\theta_{\xi,\eta}:\xi,\eta\in X\}
\]
the set of \emph{compact} operators. It is a closed two-sided ideal in $\Ll(X)$.

\begin{defn}
We say that a right Hilbert $A$-module $X$ is a \emph{$C^*$-correspondence over $A$}
when a $*$-homomorphism $\phi_X:A\to\Ll(X)$ is given. We call $\phi_X$ the
\emph{left action} of the $C^*$-correspondence.
\end{defn}

\begin{exm} \label{exm:algebracorrespondence}
Let $D$ be a $C^*$-algebra. Then there is a natural way to realise $(D,D)$
as a $C^*$-correspondence by defining the left and right actions by multiplication
and inner-product $\langle a, b\rangle = a^*b$ for $a,b\in D$. With a $C^*$-correspondence defined in this manner, there is a canonical isomorphism
\[
 \iota_D:\Kk(D) \to D
\]
satisfying $\iota_D(\theta_{a,b}) = ab^*$.
\end{exm}

For $C^*$-correspondences $(X,A)$ and $(Y,B)$ and a continuous linear map $\psi_X:X\to Y$
a $*$-homomorphism $\psi_X^+:\Kk(X) \to \Kk(Y)$ is defined satisfying
\[
 \psi_X^+(\theta_{\xi,\eta}) = \theta_{\psi_X(\xi),\psi_X(\eta)}.
\]
To see that this map is well-defined we refer the reader to \cite{Pi,KaPiWa} and
\cite[Proposition 3.10]{BaGu2}. If $(Z,C)$ is another $C^*$-correspondences and
$\psi_Y:Y\to Z$ is a continuous map, then we have the following composition law:
$ \psi_Y^+\circ\psi_X^+ = (\psi_Y\circ\psi_X)^+$.

Now we aim to define the category of $C^*$-correspondences and
introduce a covariant functor $F$ from the category of $C^*$-correspondences to the
category of $C^*$-algebras, satisfying $F(X,A)=\Oo_X$ where $\Oo_X$ is the $C^*$-algebra
associated to $(X,A)$ in the sense of Katsura \cite{Ka1}. To this end,
we first define after Katsura (\cite{Ka1}) an ideal $J_X$ of $A$ by
\[
 J_X := \{a\in A:\phi_X(a)\in\Kk(X) \mbox{ and } ab=0 \mbox{ for all } b\in \ker(\phi_X)\}.
\]
Note that $J_X=\phi_X^{-1}(\Kk(X))$ whenever $\phi_X$ is injective.

\begin{defn}\label{defn:morf}
Let $(X,A)$ and $(Y,B)$ be $C^*$-correspondences. Let $\psi_X:X\to Y$ be a linear map
and $\psi_A:A\to B$ be a $C^*$-homomorphism. We say that the pair $(\psi_X,\psi_A)$
is a \emph{morphism of $C^*$-correspondences} if the following conditions hold.
\begin{enumerate}
 \item[(C1)] \label{C1} $\langle \psi_X(\xi),\psi_X(\eta)\rangle =\psi_A(\langle \xi,\eta\rangle )$
 for all $\xi,\eta\in X$,
 \item[(C2)] \label{C2} $\psi_X(\phi_X(a)\xi) = \phi_Y(\psi_A(a))\psi_X(\xi)$ for all $\xi\in X$
 and $a\in A$,
 \item[(C3)] \label{C3} $\psi_A(J_X) \subset J_Y$, and
 \item[(C4)] \label{C4} $\phi_Y(\psi_A(a)) = \psi_X^+(\phi_X(a))$ for all $a\in J_X$.
\end{enumerate}
\end{defn}

\begin{prop}
Let $(\psi_X,\psi_A):(X,A)\rightarrow(Y,B)$ be a morphism of $C^*$-correspondences. Then
$\psi_X(\xi a) = \psi_X(\xi)\psi_A(a)$ for all $a\in A$. Furthermore, if the
image $\psi_X(X) \subset Y$ is dense, then condition (C4) automatically follows from
(C1)--(C3).
\end{prop}

\begin{proof}
The first part follows from \cite[Theorem 2.3]{BaGu3}. For the second property, suppose
the image $\psi_X(X) \subset Y$ is dense. It follows from the first
property and (C1) that for any $\theta_{\xi,\eta} \in \Kk(X)$ and $\psi_X(\zeta) \in \psi_X(X)$ we have
\[
 \psi_X^+(\theta_{\xi,\eta})\psi_X(\zeta) = \psi_X(\xi)\langle\psi_X(\eta),\psi_X(\zeta)\rangle =
 \psi_X(\theta_{\xi,\eta} \zeta).
\]
Therefore $\psi_X^+(K)\psi_X(\zeta) = \psi_X(K \zeta)$ for all $K\in\Kk(X)$. In particular,
for all $a\in J_X$ we have
$$ \psi^+_X(\phi_X(a))\psi_X(\zeta)=\psi_X(\phi_X(a)\zeta)=\phi_Y(\psi_A(a))\psi_X(\zeta) $$
by (C2), since $\phi_X(a)\in\Kk(X)$. By density of $\psi_X(X) \subset Y$, we may now conclude that
$\psi^+_X(\phi_X(a) = \phi_Y(\psi_A(a))$.
\end{proof}

The following simple example illustrates the fact that (C4) may fail if the image
of $\psi_X$ is not dense.

\begin{exm}
Let $X$ be the one dimensional Hilbert space with generator $e$, and $Y$ be
the two dimensional Hilbert space with generators $f_1$ and $f_2$. Let $A = B = \CC$.
Then by defining left and right actions as multiplication, $(X,A)$ and $(Y,B)$
are $C^*$-correspondences. Define a pair of maps $(\psi_X,\psi_A):(X,A) \to (Y,B)$ by
$\psi_X(e) = f_1$ and $\psi_A = \id$. Clearly, the image of $\psi_X$ is not dense in $Y$.
It is easily shown that conditions (C1), (C2) and (C3) hold. But (C4) says that we should have
\begin{eqnarray*}
 & & \psi_X^+(\phi_X(1)) = \phi_Y(\psi_A(1)) \\
 &\iff& \psi_X^+(\theta_{e,e}) = \phi_Y(1) \\
 &\iff& \theta_{f_1,f_1} = \theta_{f_1,f_1} + \theta_{f_2,f_2}
\end{eqnarray*}
which would imply that the generator $f_2$ is zero. So (C4) does not hold.
\end{exm}

We can now construct the desired category, which we call $\Cc$. The objects are given by
\[
 \Obj(\Cc) = \{(X,A):X \mbox{ is a $C^*$-correspondence over $A$}\}
\]
and morphisms $\Mor(\Cc)$ as in Definition \ref{defn:morf}.

There are well defined domain and codomain maps $\Mor(\Cc)\to \Obj(\Cc)$, namely for
$(\psi_X,\psi_A):(X,A)\to(Y,B)$, we have $\dom(\psi_X,\psi_A)=(X,A)$ and
$\cod(\psi_X,\psi_A)=(Y,B)$. It is clear from the definition of morphisms that the
composition of two composable morphisms will result in another morphism. For an object
$(X,A)$ the identity morphism is $\id_{(X,A)}=(\id_X,\id_A)$ which is also clearly a
morphism of $C^*$-correspondences.

\par The next step is to define the $C^*$-algebra associated to a $C^*$-correspondence,
and show that this process is naturally implemented by a functor between the categories.
In order to do this, we first need to define what we mean by a covariant representation of
a $C^*$-correspondence on a $C^*$-algebra. We do this via morphisms of $C^*$-correspondences.

\begin{defn} \label{defn:representations}
Let $(X,A)$ be a $C^*$-correspondence and let $D$ be an arbitrary $C^*$-algebra.
A \emph{covariant representation} of $(X,A)$ on $D$ is a morphism $(\rho_X,\rho_A):(X,A)
\to (D,D)$, where $(D,D)$ is the $C^*$-correspondence introduced in Example
\ref{exm:algebracorrespondence}.
\end{defn}

\begin{rmk} \label{rmk:representations}
It is easy to see that this definition is equivalent to that of a covariant representation, given in \cite{Ka1}. If we consider morphisms that only satisfy conditions (C1) and (C2), then we recover the original definition of a representation of a $C^*$-correspondence. Furthermore, it follows from the  somorphism $\Kk(D) \cong D$ that $J_D = D$, and hence condition (C3) is automatic for morphisms of this form. Condition (C4) is the covariance condition of \cite{Ka1}; i.e. it says that
\[
 \rho_X^+(\phi_X(a)) = \iota_D^{-1}(a)
\]
where $\iota_D:\Kk(D) \to D$ is the isomorphism from example \ref{exm:algebracorrespondence}.
\end{rmk}

\begin{defn} \cite[Definition 2.6]{Ka1}
Define the $\Oo_X$ to be the universal $C^*$-algebra generated by the image of $(X,A)$
under the universal covariant representation $(\pi_X,\pi_A)$.
\end{defn}

It is easy to show the existence of such a universal representation, which in fact 
is always injective.

Now, we want to construct a functor $F$ from our category $\Cc$ to the category of
$C^*$-algebras, satisfying $F(X,A)=\Oo_X$. In order to do this, we must first
see that any morphism $(\psi_X,\psi_A):(X,A)\to (Y,B)$ extends to a unique $C^*$-homomorphism,
which we denote by $\Psi:\Oo_X\to\Oo_Y$.

\begin{prop} \label{prop:morphismextension} Let $(\psi_X,\psi_A):(X,A)\to (Y,B)$ be a
morphism of $C^*$-correspondences. Then this morphism extends to a unique $C^*$-homomorphism
$\Psi:\Oo_X\to\Oo_Y$ such that the following diagram commutes.
\[
 \beginpicture
  \setcoordinatesystem units <2cm,2cm>
  \put{$(X,A)$} at 0 1.5
  \put{$(Y,B)$} at 2 1.5
  \put{$\Oo_X$} at 0 0
  \put{$\Oo_Y$} at 2 0
  \arrow <0.15cm> [0.25,0.75] from 0 1.3 to 0 0.2
  \arrow <0.15cm> [0.25,0.75] from 0.4 1.5  to 1.6 1.5
  \arrow <0.15cm> [0.25,0.75] from 2 1.3 to 2 0.2
  \arrow <0.15cm> [0.25,0.75] from 0.2 0 to 1.8 0
  \put{$\Psi$} at 1 0.2
  \put{$(\pi_X,\pi_A)$} at -0.4 0.75
  \put{$(\psi_X,\psi_A)$} at 1 1.7
  \put{$(\pi_Y,\pi_B)$} at 2.4 0.75
 \endpicture
\]
\end{prop}

\begin{proof}
It follows from the definition that $(\pi_Y \circ \psi_X, \pi_B \circ \psi_A)$ is a
covariant representation of $(X,A)$ on $\Oo_Y$. Hence the universal property of $\Oo_X$
implies the existence of the required $*$-homomorphism $\Psi:\Oo_X\to\Oo_Y$ such that the
diagram commutes.
\end{proof}

\begin{defn}
Define a map $F$ from the category $\Cc$ of $C^*$-correspondences to the category
of $C^*$-algebras which satisfies
\[
 F(X,A) = \Oo_X
\]
and for a morphism $(\psi_X,\psi_A):(X,A)\to(Y,B)$ we let $F(\psi_X,\psi_A)$ be the
unique map $\Psi:\Oo_X\to\Oo_Y$ satisfying the conditions of Proposition \ref{prop:morphismextension}.
\end{defn}

\begin{prop}
The map $F$ is a covariant functor from the category of $C^*$-correspon-dences to
the category of $C^*$-algebras.
\end{prop}

\begin{proof}
First fix $C^*$-correspondences $(X,A), (Y,B)$ and $(Z,C)$, and morphisms of
$C^*$-correspondences $(\psi_X,\psi_A):(X,A)\to(Y,B)$ and $(\omega_Y,\omega_B):(Y,B)\to(Z,C)$.
Since $\id_X$ is linear, $\id_A$ is a homomorphism and $\id_A(J_X)=J_X$, we know that
$(\id_X,\id_A)$ is a morphism of $C^*$-correspondences. So in order to see that $F$ is
a covariant functor, we need to show that
\begin{enumerate}
 \item \label{id} $\id_{F(X,A)} = F(\id_{(X,A)})$; and
 \item \label{comp} $F(\psi_X,\psi_A)\circ F(\omega_Y,\omega_B) =
 F((\psi_X,\psi_A)\circ(\omega_Y,\omega_B))$.
\end{enumerate}
For (\ref{id}), we need only show that this holds on the generators
$\{\pi_X(\xi):\xi\in X\}$ and $\{\pi_A(a):a\in A\}$ and this follows from the
commutativity of the diagram in Proposition \ref{prop:morphismextension}.

For (\ref{comp}), again we only see that it holds on generators, so fix $\xi\in X$
and $a\in A$. Then by three applications of the commutativity of the diagram
in Proposition \ref{prop:morphismextension} and the fact that the composition
$(\psi_X,\psi_A)\circ(\omega_Y,\omega_B)$ is a morphism of $C^*$-correspondences we have
\begin{eqnarray*}
 F(\omega_Y,\omega_B)\circ F(\psi_X,\psi_A)(\pi_X(\xi)) &=&
 F(\omega_Y,\omega_B)\circ \pi_Y((\psi_X,\psi_A)(\xi)) \\
 &=& \pi_Z((\omega_Y,\omega_B)\circ(\psi_X,\psi_A)(\xi)) \\
 &=& F((\omega_Y,\omega_B)\circ(\psi_X,\psi_A))(\pi_X(\xi))
\end{eqnarray*}
and similarly we can show that
\[
 F(\omega_Y,\omega_B)\circ F(\psi_X,\psi_A)(\pi_X(a)) =
 F((\omega_Y,\omega_B)\circ(\psi_X,\psi_A))(\pi_X(a))
\]
as required. So $F$ is a covariant functor from the category of $C^*$-correspondences
to the category of $C^*$-algebras.
\end{proof}

In what follows we will just write $\Oo_X$ for $F(X,A)$ and use capitalised Greek
characters for induced homomorphisms between $C^*$-algebras.

\section{Gluing $C^*$-correspondences} \label{sec:gluing}

The purpose of this section is to show that the functor $F$ constructed above is
well-behaved with respect to taking pullbacks. This will be useful in applications when
we consider noncommutative spaces as being constructed by gluing two underlying spaces
together over a common boundary. The idea of taking pullbacks on the level of
$C^*$-correspondences motivates the following definition, which is due to Baki\'{c}
and Gulja\v{s}, \cite{BaGu2}.

\begin{defn} \label{defn:pullback}
Given $(X,A),(Y,B)$ and $(Z,C) \in \Cc$, and morphisms of $C^*$-correspon-dences
$(\psi_X,\psi_A):(X,A)\to(Z,C)$, $(\omega_Y,\omega_B):(Y,B)\to(Z,C)$, define
the \emph{restricted direct sum}
\[
 X \dsz Y:= \{(\xi,\eta)\in X \oplus Y:\psi_X(\xi)=\omega_Y(\eta)\}.
\]
\end{defn}

\begin{prop} \label{prop:restricteddirectsum}
The restricted direct sum $X\dsz Y$ is a $C^*$-correspondence over the 
$C^*$-algebra $A\dsc B$ defined to be the pullback $C^*$-algebra of $A$ and 
$B$ along $\psi_A$ and $\omega_B$.
\end{prop}

\begin{proof}
It follows from \cite[Lemma 2.1]{BaGu2} that $X\dsz Y$ is a Hilbert $A\dsc B$-module. 
In order to prove that it is also a $C^*$-correspondence we need a left action 
$\phi_{X\dsz Y}$. For $(\xi,\eta)\in X\dsz Y$ and $(a,b)\in A\dsc B$ we define
\[
 \phi_{X\dsz Y}(a,b)(\xi,\eta) := (\phi_X(a)\xi,\phi_Y(b)\eta).
\]
To see that this is an element of $X\dsz Y$, first recall that 
$(\pi_Z\circ\psi_A,t_Z\circ\psi_X)$ is a covariant representation of 
$(X,A)$ on $\Oo_Z$ and $(\pi_Z\circ\omega_Y,\pi_C\circ\omega_B)$ is a covariant 
representation of $(Y,B)$ on $\Oo_Z$. Then
\begin{eqnarray*}
 \pi_C\circ\psi_A(\phi_X(a)\xi) &=& (\pi_C\circ\psi_A)(a)(\pi_Z\circ\psi_X)(\xi) \\
 &=& (\pi_C\circ\omega_B)(b)(\pi_Z\circ\omega_Y)(\eta) \\
 &=& \pi_C(\omega_B(\phi_Y(b)\eta))
\end{eqnarray*}
and then since $\pi_Z$ is injective, it follows that $\psi_A(\phi_X(a)\xi)=\omega_B(\phi_Y(b)\eta)$. It is clear from the defintion that it is a $*$-homomorphism so $\phi_{X\dsz Y}:A \dsc B \to \Ll(X\dsz Y)$ is a left-action and $(X\dsc Y,A\dsc B)$ is a $C^*$-correspondence.
\end{proof}

The following theorem is the main result of this section.

\begin{theorem} \label{thm:pullback} Let $(X,A),(Y,B)$ and $(Z,C)$ be $C^*$-correspondences. 
Fix morphisms of $C^*$-correspondences $(\psi_X,\psi_A):(X,A)\to(Z,C)$, 
$(\omega_Y,\omega_B):(Y,B)\to(Z,C)$ satisfying
\begin{enumerate}
\item $(\psi_X,\psi_A)$ and $(\omega_Y,\omega_B)$ are surjective morphisms with $\psi_A(\ker(\phi_X)) = \omega_B(\ker(\phi_Y))$,
\item $\phi_X(A) \subset \Kk(X)$ and $\phi_Y(B) \subset \Kk(Y)$, and
\item the ideals $\ker(\phi_X)$ and $\ker(\phi_Y)$ are complemented; i.e. there exist ideals $J_A \unlhd A$ and $J_B \unlhd B$ such that $A = J_A \oplus \ker(\phi_X)$ and $B = J_B \oplus \ker(\phi_Y)$.
\end{enumerate}
Then
\[
 \Oo_{X\dsz Y} \cong \Oo_X \oplus_{\Oo_Z} \Oo_Y
\]
where $\Oo_X \oplus_{\Oo_Z} \Oo_Y$ is the pullback $C^*$-algebra of $\Oo_X$ and $\Oo_Y$ along $\Psi$ and $\Omega$.
\end{theorem}

We begin by showing that there is a covariant representation of the restricted direct sum correspondence $(X \dsz Y, A \dsc B)$ on the pullback $C^*$-algebra $\Oo_X \oplus_{\Oo_Z} \Oo_Y$. First, we need a definition.

\begin{defn}
Let $(\rho_X,\rho_A)$ be a covariant representation of a $C^*$-correspondence $(X,A)$ on a $C^*$-algebra $D$. Then we say that $(\rho_X,\rho_A)$ \emph{admits a gauge action} if for $z\in\TT$ there exists a $*$-homomorphism $\alpha_z:C^*(\rho_X,\rho_A)\to C^*(\rho_X,\rho_A)$ such that $\alpha_z(\rho_A(a))=\rho_A(a)$ and $\alpha_z(\rho_X(\xi))=z \, \rho_X(\xi)$ for all $a\in A$ and $\xi\in X$. We say an ideal $I\subset C^*(\rho_X,\rho_A)$ is \emph{gauge invariant} if $\alpha_z(I)\subset I$ for all $z\in\TT$.
\end{defn}

It is well-known that the universal covariant representation $(\pi_X,\pi_A)$ of $(X,A)$ admits a gauge action.

\begin{lemma} \label{lemma:kernelideal}
Let $(\psi_X,\psi_A)$ be a morphism of $C^*$-correspondences, with associated $C^*$-homomorphism $\Psi:\Oo_X \to \Oo_Y$. Then $\ker(\Psi)$ is a gauge invariant ideal of $\Oo_X$. Furthermore, if we assume that $(\psi_X,\psi_A)$ is surjective, $\phi_X(A) \subset \Kk(X)$ and $\ker(\phi_X)$ is complemented then $\ker(\Psi)$ is the ideal generated by $\pi_X(\ker(\psi_A))$.
\end{lemma}

\begin{proof}
Let $\gamma : \TT \to \Aut(\Oo_X)$ be the gauge action. To see that $\ker(\Psi)$ is a gauge invariant ideal it is enought to show that $\Psi$ commutes with the automorphism $\gamma_z$ for any $z \in \TT$. This is easily seen to hold the generators $\{t_X(\xi):\xi \in X\}$ and $\{\pi_X(a):a \in A\}$.

Now suppose the morphism satisfies the extra hypotheses. We begin by showing that the ideal $\ker(\psi_A)$ is $X$-invariant and $X$-saturated in the sense of \cite{MuTo}. That is, we need to show that
\begin{enumerate}
\item $\phi_X(\ker(\psi_A)) X \subset X \ker(\psi_A)$, and
\item $a \in J_X$ and  $\phi_X(a) X \subset X\ker(\psi_A) \Longrightarrow a \in \ker(\psi_A)$.
\end{enumerate}
For (1), suppose $\xi \in X \ker(\psi_A)$. Then \cite[Proposition 1.3]{Ka2} implies that $\langle \xi,\xi \rangle \in \ker(\psi_A)$, so we have $\psi_X(\xi) = 0$ by definition of a morphism. Likewise, if $\xi \in \ker(\psi_X)$, then $\langle\xi,\xi\rangle \in \ker(\psi_A)$ so \cite[Proposition 1.3]{Ka2} implies $\xi \in X \ker(\psi_A)$. So we have $\ker(\psi_X) = X\ker(\psi_A)$ and condition (1) easily follows.

For condition (2), fix $a \in J_X$ and suppose $\phi_X(a) X \subset X \ker(\psi_A)$. From the argument above this is equivalent to $\psi_X(\phi_X(a) X) = \{0\}$. Since $\psi_X$ is surjective, this means we must have $\phi_Y(\psi_A(a)) = 0$ and hence $\psi_A(a) = 0$ since $\psi_A(a) \in J_Y$ and $\phi_Y$ is injective on $J_Y$. So $a \in \ker(\psi_A)$.

Finally, \cite[Theorem 6.4]{MuTo} implies that $\ker(\Psi)$ is the ideal generated by $\pi_X(\pi_X^{-1}(\ker(\Psi))) = \pi_X(\ker(\psi_A))$ as required.
\end{proof}

\begin{defn}
Define a pair of maps $(\rho_{X\dsz Y},\rho_{A\dsc B}):(X \dsz Y, A\dsc B) \to \Oo_X \oplus_{\Oo_Z} \Oo_Y$ as follows. The map $\rho_{A\dsc B}$ satisfies
\[
 \rho_{A\dsc B}(a,b) = (\pi_A(a),\pi_B(b))
\]
and $\rho_{X\dsz Y}$ satisfies
\[
 \rho_{X\dsz Y}(\xi,\eta) = (\pi_X(\xi),\pi_Y(\eta)).
\]
\end{defn}

\begin{prop} The pair $(\rho_{X\dsz Y},\rho_{A\dsc B})$ is an injective covariant
representation of the restricted direct sum $(X\dsz Y,A\dsc B)$ on the pullback
$C^*$-algebra $\Oo_X \oplus_{\Oo_Z} \Oo_Y$. Furthermore, this representation admits a gauge action.
\end{prop}
\begin{proof}
It is easy to see from the definition that $\rho_{A\dsc B}$ is a $*$-homomorphism 
and that $\rho_{X\dsz Y}$ is a linear map. Furthermore, routine calculations show that 
conditions (C1) and (C2) are satisfied. We know $\rho_{A\dsc B}$ is injective because 
the universal homomorphisms $\pi_X$ and $\pi_Y$ are injective. Remark \ref{rmk:representations}  
implies that (C3) is automatic in this case. For (C4), since $\pi_X$ and $\pi_Y$ are covariant representations, it is enough to show that $(a,b) \in J_{X \dsz Y}$ implies $a \in J_X$ and $b \in J_Y$. First notice that we have an inclusion $\Kk(X \dsz Y) \subset \Kk(X) \oplus \Kk(Y)$ given on generators by
\[
 \theta_{(\xi,\eta),(\xi',\eta')} \mapsto (\theta_{\xi,\xi'}, \theta_{\eta,\eta'}).
\]
So $(a,b) \in J_{X \dsz Y}$ implies $\phi_X(a) \in \Kk(X)$ and $\phi_Y(b) \in \Kk(Y)$. Now, fix $c \in \ker(\phi_X)$. Then since $\psi_A(\ker(\phi_X)) = \{0\} = \omega_B(\ker(\phi_Y))$ we can find an element $d \in \ker(\phi_Y)$ so that $(c,d) \in \ker(\phi_{X \dsz Y})$. Hence
\[
 (a,b) (c,d) = 0 \Longrightarrow a c = 0
\]
and we see that $a \in J_X$. Likewise, we may show that $b \in J_Y$ and we have an injective covariant representation as required.

Finally, it is easy to show that $(\rho_{X\dsz Y},\rho_{A\dsc B})$ admits a gauge action 
simply because both of the universal representations $\pi_X$ and $\pi_Y$ admit gauge actions.
\end{proof}

So we have an injective covariant representation of $(X\dsz Y,A\dsc B)$ on the pullback 
$C^*$-algebra $\Oo_X \oplus_{\Oo_Z} \Oo_Y$ so in particular, the universality of $\Oo_{X\dsz Y}$ 
implies that there exists a homomorphism $P:\Oo_{X\dsz Y}\to\Oo_X \oplus_{\Oo_Z} \Oo_Y$. We 
want to see that this is an isomorphism. To do this, we use the gauge-invariant uniqueness 
theorem, originally proved for the injective left-action case in \cite{FoMuRa} and then in the general case by Katsura, \cite[Theorem 6.4]{Ka3}. It is restated without proof here for the readers convenience.

\begin{theorem} \cite[Theorem 6.4]{Ka3} \label{thm:gaugeinvariantuniquenesstheorem}
For a covariant representation $(\rho_X,\rho_A)$ of a $C^*$-correspondence $(X,A)$, the 
$*$-homomorphism $P:\Oo_X\to C^*(\rho_X,\rho_A)$ is an isomorphism if and only if 
$(\rho_X,\rho_A)$ is injective and admits a gauge action.
\end{theorem}

We now have the required results to begin the proof of Theorem \ref{thm:pullback}.

\begin{proof}[Proof of Theorem \ref{thm:pullback}]
We have an injective, covariant representation $\rho = (\rho_{X\dsz Y},\rho_{A\dsc B})$ of $X\dsz Y$ on the pullback $C^*$-algebra $\Oo_X \oplus_{\Oo_Z} \Oo_Y$. So the universality of $\Oo_{X \dsz Y}$ induces a homomorphism $\phi : \Oo_{X \dsz Y} \to \Oo_X \oplus_{\Oo_Z} \Oo_Y$. Let $\alpha, \beta$ be the gauge actions on $\Oo_X$ and $\Oo_Y$ respectively. Then clearly we get a gauge action $\alpha \oplus \beta$ on the pullback $\Oo_X \oplus_{\Oo_Z} \Oo_Y$ that is compatible with $\rho$. So the gauge invariant uniqueness theorem implies that $\phi$ is in fact an injection. It only remains to be seen that this map is a surjection.

Fix $(x,y)\in \Oo_X\oplus_{\Oo_Z} \Oo_Y$. Since the $C^*$-correspondence morphisms
$(\psi_X,\psi_A)$ and $(\omega_Y,\omega_B)$ are surjective, we can find an element $(x,y')\in \Oo_{X\dsz Y}$ with $y'\in\Omega^{-1}(\Psi(x))$. By
definition of this element, we must have $\Omega(y)=\Omega(y')=\Psi(x)$ and hence
$y-y'\in \ker(\Omega)$. Hence, we can write $(x,y)$ as a sum
\[
 (x,y) = (x,y') + (0,w)
\]
with $(x,y')\in \Oo_{X\dsz Y}$ and $w=y-y'\in \ker(\Omega)$. So to see that $(x,y)\in \Oo_{X\dsz Y}$,
we need only show that $(0,w)\in \Oo_{X\dsz Y}$. It is enough to show that $\ker(\Omega)$ is the ideal generated by $\pi_Y(\ker(\omega_B))$, and this follows from Lemma \ref{lemma:kernelideal}.

So we have the required isomorphism.
\end{proof}

\section{Construction of $2n$-dimensional mirror quantum spheres} \label{sec:mirrorsphere}

As an application of Theorem \ref{thm:pullback}, in this section we present a realisation of the $2n$ dimensional mirror quantum sphere $C(S^{2n}_{q,\beta})$ where $q\in(0,1)$, as a $C^*$-algebra associated to a $C^*$-correspondence. For the original construction and results on this object we refer the reader to \cite{HoSz}. We recall the basic definition here for the convenience of the reader.

Firstly, fix $n\in\NN$. We construct the $2n$ dimensional quantum sphere by gluing two $2n$ dimensional quantum discs along their boundary, a $2n-1$ dimensional quantum sphere. The `mirroring' is obtained by pre-composing one of the discs with an automorphism that `flips' it prior to gluing. More concretely, we construct it as the pullback of the following diagram
\[
 \beginpicture
  \setcoordinatesystem units <2cm,2cm>
  \put{$C(\mathbb{D}_q^{2n})$} at -1 0
  \put{$C(\mathbb{D}_q^{2n})$} at 1 1
  \put{$C(S_q^{2n-1})$} at 1 0
  \put{$\pi$} at 0 0.1
  \put{$\beta\circ\pi$} at 1.3 .5
  \arrow <0.15cm> [0.25,0.75] from -0.65 0 to 0.55 0
  \arrow <0.15cm> [0.25,0.75] from 1 0.8 to 1 0.2
 \endpicture
\]
where $\pi:C(\mathbb{D}_q^{2n})\to C(S_q^{2n-1})$ is the natural surjection and $\beta\in\Aut (C(S_q^{2n-1}))$.

The first step is to represent each $C^*$-algebra as a $C^*$-algebra generated by a $C^*$-correspondence. We know from \cite{HoSz} that $C(\DD_q^{2n})$ is isomorphic to the graph algebra $C^*(M_n)$, where $M_n$ is the graph with $n+1$ vertices $M_n^0 = \{v_1,\dots,v_{n+1}\}$, edge set $M_n^1 = \{e_{i,j}:1\leq i \leq n, i\leq j \leq n+1\}$ and range and source maps $r$ and $s$ satisfying $s(e_{i,j}) = v_i \mbox{ and } r(e_{i,j}) = v_j$.

Now, consider the vector space $X=\opsp\{w_{i,j}:i=1,\dots,n, j=i,\dots,n+1\}$ and $C^*$-algebra generated by mutually orthogonal projections $A=\opsp\{P_i:i=1,\dots,n+1\}$. Define an $A$ valued inner-product on $X$ satisfying
\[
 \langle w_{i,j},w_{k,l}\rangle  = \delta_{i,k} \delta_{j,l} P_j
\]
and a right action of $A$ on $X$ satisfying
\[
 w_{i,j} P_k = \delta_{j,k} w_{i,j}.
\]
Define a left action $\phi_X:A\to\Ll(X)$ satisfying
\[
 \phi_X(P_k)w_{i,j} = \delta_{i,k} w_{i,j}.
\]
Then it is easily checked that $X \cong C_c(M_n^1)$ and $A \cong C_0(M_n^0)$ and the $C^*$-correspondence structure matches that defined in \cite[Section 3.4]{Ka1}, so \cite[Proposition 3.10]{Ka1} implies that we have an isomorphism $\Oo_{X} \cong C(\DD_q^{2n})$.

Similarly for $C(S_q^{2n})$, consider the vector space
\[
 Z=\opsp\{z_{i,j}:i=1,\dots,n,j=i,\dots,n\}
\]
and $C^*$-algebra generated by mutually orthogonal projections
\[
 C=\opsp\{S_i:i=1,\dots,n\} \cong \mathbb{C}^n.
\]
Then $\Oo_{Z} \cong C(S_q^{2n-1})$. We refer the reader to \cite{HoSz} for the details of the directed graph that we are using to construct $Z$ and $C$.

Now, we want to see that there exists a morphism of $C^*$-correspondences $(\psi_X,\psi_A):(X,A)\to (Z,C)$ such that the induced homomorphism $\Psi:\Oo_X \to \Oo_Z$ and the surjection $\pi:C(\DD_q^{2n})\to C(S_q^{2n})$ are intertwined by the isomorphism $\Oo_X \cong C(\DD_q^{2n})$ and the isomorphism $\Oo_Z \cong C(S_q^{2n})$. Define
\[
 \psi_X(w_{i,j}) = \left\{ \begin{array}{ll}
 z_{i,j} & \mbox{ for } 1\leq i \leq n \mbox{ and } i\leq j\leq n \\
 0 & \mbox{ if } i=n, j=n+1
 \end{array} \right.
\]
and
\[
 \psi_A(P_i) = \left\{ \begin{array}{c}
 S_i \ \mbox{ for } 1\leq i \leq n \\
 0 \ \mbox{ if } i=n+1
 \end{array} \right.
\]
Then it is routine to show that this is a morphism of $C^*$-correspondences and the induced map $\Psi:\Oo_X\to\Oo_Z$ is the desired homomorphism. So we have
\[ \begin{array}{l}
 \Psi(\pi_A(w_{i,j})) = \pi_Z(z_{i,j}) \mbox{ for } i=1,\dots,n,j=i,\dots,n \\
 \Psi(\pi_A(w_{n,n+1})) = 0 \\
 \Psi(\pi_A(P_i)) = \pi_C(S_i) \mbox{ for } i=1,\dots,n \\
 \Psi(\pi_A(P_{n+1})) = 0.
\end{array} \]

For the mirroring, we compose with the automorphism $\beta\in\Aut(\Oo_Z)$ satisfying
\[ \begin{array}{l}
 \beta(\pi_Z(z_{i,j})) = \pi_Z(z_{i,j}) \mbox{ for } i=1,\dots,n-1,j=i,\dots,n-1  \\
 \beta(\pi_Z(z_{n,n})) = \pi_Z(z_{n,n})^* \\
 \beta(\pi_C(S_i)) = \pi_C(S_i)
\end{array} \]

However, we are not yet ready to use Theorem \ref{thm:pullback} because the map $\beta\circ\Psi$ does not come from a $C^*$-correspondence morphism from $(X,A)$ to $(Z,C)$. So we must find another $C^*$-correspondence $(Y,B)$ and a $C^*$-correspondence morphism $(\omega_Y,\omega_B):(Y,B)\to(Z,C)$ such that $\Oo_Y\cong C(\mathbb{D}_q^{2n})$ and the extension $\Omega:\Oo_Y\to\Oo_Z$ of $(\omega_Y,\omega_B)$ is the $C^*$-homomorphism $\beta\circ\Psi$.

Let $Y_1:=\{x_{i,j}:1\leq i \leq n-1, i\leq j \leq n+1\}, Y_2:=\{x_{i,n,j}:1\leq i \leq n-1, j\geq 1\}$ and $Y_3:=\{y,y',y_i:i\geq 1\}$ be sets of generators and define a vector space
\[
 Y:=\clsp(Y_1\cup Y_2\cup Y_3).
\]
Let $B$ be the $C^*$-algebra generated by a set of $n+1$ mutually orthogonal non-zero projections $\{R_i:i=1,\dots,n+1\}$ and a set  $\{Q_j:j\geq 1\}$ of mutually orthogonal projections satisfying $Q_j\leq R_n$ for all $j\geq 1$.

Define a $B$ valued inner-product on $Y$ by
\[
 \begin{array}{ll}
  \langle x_{i,j},x_{k,l}\rangle  = \delta_{i,k}\delta_{j,l} R_j & \ \ \ \langle x_{i,j},x_{k,n,l}\rangle  = \delta_{i,k}\delta_{j,n} Q_l \\
  \langle x_{i,j},y\rangle  = 0 & \ \ \ \langle x_{i,j},y_k\rangle  = 0 \\
  \langle x_{i,n,j},y\rangle  = 0 & \ \ \ \langle x_{i,n,j},y_k\rangle  = 0 \\
  \langle x_{i,n,j},x_{k,n,l} \rangle = \delta_{i,k}\delta_{j,l} Q_j & \ \ \ \langle y,y\rangle  = R_n \\
  \langle y,y'\rangle  = Q_1 & \ \ \ \langle y,y_i\rangle  = Q_{i+1} \\
  \langle y',y'\rangle  = Q_1 & \ \ \ \langle y',y_i\rangle  = 0 \\
  \langle y_i,y_j\rangle  = \delta_{i,j} Q_{i+1}
 \end{array}
\]
and a right action of $B$ on $Y$ given by
\[
 \begin{array}{ll}
  x_{i,j} R_k = \delta_{j,k} x_{i,j} & \ \ \ x_{i,j} Q_k = \delta_{j,n} x_{i,n,k} \\
  x_{i,n,j} R_k = \delta_{n,k} x_{i,n,j} & \ \ \ x_{i,n,j} Q_k = \delta_{j,k} x_{i,n,j} \\
  y R_k = \delta_{n,k} y & \ \ \ y Q_k = \left\{ \begin{array}{l} y_{k-1} \mbox{ if } k\geq 2 \\ y' \mbox{ if } k=1 \\ 0 \mbox{ otherwise} \end{array} \right. \\
  y' R_k = \delta_{n,k} y' & \ \ \ y' Q_k = \delta_{1,k} y' \\
  y_i R_k = \delta_{n,k} y_i & \ \ \ y_i Q_k = \delta_{i+1,k} y_i.
 \end{array}
\]
We can define a left-action $\phi_Y:B\to\Ll(Y)$ by
\[
 \begin{array}{ll}
  \phi_Y(R_k)x_{i,j} = \delta_{i,k} x_{i,j} & \ \ \ \phi_Y(R_k)x_{i,n,j} = \delta_{i,k} x_{i,n,j} \\
  \phi_Y(R_k)y = \delta_{n,k} (y-y')+ \delta_{n+1,k} y' & \ \ \ \phi_Y(R_k)y' = \delta_{n+1,k} y' \\
  \phi_Y(R_k)y_i = \delta_{n,k} y_i & \ \ \ \phi_Y(Q_k)x_{i,j} = 0 \\
  \phi_Y(Q_k)x_{i,n,j} = 0 & \ \ \ \phi_Y(Q_k)y = y_k \\
  \phi_Y(Q_k)y' = 0 & \ \ \ \phi_Y(Q_k)y_i = \delta_{i,k}y_i.
 \end{array}
\]
Tedious, but straightforward calculations show that $(Y,B)$ is a $C^*$-correspondence.

\begin{theorem} \label{thm:XYisomorphism}
The $C^*$-algebras $\Oo_X$ and $\Oo_Y$ are isomorphic.
\end{theorem}

Before we can prove Theorem \ref{thm:XYisomorphism} we need some preliminary results.

\begin{lemma} \label{lemma:idealinA}
The ideal $J_X$ is generated by the projections $P_i$ for $1\leq i \leq n$. Furthermore, we have
\begin{enumerate}
 \item $\phi_X(P_i) = \sum_{j=i}^n \theta_{w_{i,j},w_{i,j}}$ for $1\leq i\leq n-1$; and
 \item $\phi_X(P_n) = \theta_{w_{n,n},w_{n,n}} + \theta_{w_{n,n+1},w_{n,n+1}}$.
\end{enumerate}
\end{lemma}

\begin{proof}
We can immediately see from the definition of $\phi_X$ that $\ker(\phi_X)$ is generated by $P_{n+1}$, so it will suffice to show (1) and (2). For $1\leq i\leq n-1$ have
\begin{eqnarray*}
 \left(\sum_{j=i}^n \theta_{w_{i,j},w_{i,j}}\right)(w_{k,l}) &=& \sum_{j=i}^n w_{i,j}\langle w_{i,j},w_{k,l}\rangle  \\
 &=& \sum_{j=i}^n \delta_{i,k}\delta_{j,l} w_{i,j} P_j \\
 &=& \delta_{i,k} w_{k,l} \\
 &=& \phi_X(P_i) w_{k,l}.
\end{eqnarray*}
So since $ \phi_X(P_i)$ and $\sum_{j=i}^n \theta_{w_{i,j},w_{i,j}}$ agree on generators, they must be the same operator.
Similarly, $(\theta_{w_{n,n},w_{n,n}} + \theta_{w_{n,n+1},w_{n,n+1}})w_{i,j} = \phi_X(P_n)w_{i,j}$ as required.
\end{proof}

\begin{lemma} \label{lemma:idealinB}
The ideal $J_Y$ is equal to $B$. Furthermore we have the following relations:
\begin{enumerate}
 \item For $1\leq i \leq n-1$ we have $\phi_Y(R_i) = \displaystyle\sum_{j=i}^n \theta_{x_{i,j},x_{i,j}}$
 \item $\phi_Y(R_n) = \theta_{y-y',y-y'}$
 \item $\phi_Y(R_{n+1}) = \theta_{y',y'}$
 \item For $i\geq 1$ we have $\phi_Y(Q_i) = \theta_{y_i,y_i}$.
\end{enumerate}
\end{lemma}

\begin{proof}
It is not hard to see from the definition that $\phi_Y$ is injective, so to see that $J_Y=B$ we need only show that each generator is an element of $\Kk(Y)$. So it suffices to show that the relations (1) to (4) hold. This is straightforward verification again.
\end{proof}

\begin{proof}[Proof of Theorem \ref{thm:XYisomorphism}] We construct an isomorphism $\Pi_X:\Oo_X \to \Oo_Y$ by showing that there exists an injective covariant representation $(\rho_X,\rho_A)$ of $(X,A)$ on $\Oo_Y$ and an injective covariant representation $(\rho_Y,\rho_B)$ of $(Y,B)$ on $\Oo_X$ such that the induced homomorphisms are bijective and mutually inverse.

Define the linear map $\rho_X$ on the generators of $X$ by
\begin{align*}
 \rho_X(w_{i,j}) &= \pi_Y(x_{i,j}) \mbox{ for } i=1,\dots,n-1, j=i,\dots,n+1 \\
 \rho_X(w_{n,n}) &= \pi_Y(y)^*-\pi_Y(y')^* \\
 \rho_X(w_{n,n+1}) &= \pi_Y(y')^*
\end{align*}
and define the homomorphism $\rho_A$ on the generators of $A$ by
\begin{align*}
 \rho_A(P_i) &= \pi_B(R_i) \mbox{ for } 1\leq i\leq n-1, \ \ \ \rho_A(P_n)=\pi_B(R_n), \ \ \ \rho_A(P_{n+1}) = \pi_B(R_{n+1})
\end{align*}
Routine caculations show that $(\rho_X,\rho_A):(X,A)\to \Oo_Y$ satisfy (C1) and (C2). We must show (C4). We know from Lemma \ref{lemma:idealinA} that $J_X$ is generated by the projections $P_i$ where $1\leq i \leq n$. So we need to show that we have
\[
 \rho_A(P_i) = \rho_X^+(\phi_X(P_i))
\]
for $1\leq i\leq n$. This also follows easily from Lemma \ref{lemma:idealinA}. So we have a covariant representation $(\rho_X,\rho_A)$ of $(X,A)$ on $\Oo_Y$.

Now, we want to construct another covariant representation of $(Y,B)$ on $\Oo_X$. Define a linear map $\rho_Y:Y\to \Oo_X$ by
\begin{eqnarray*}
  \rho_Y(x_{i,j}) &=& \pi_X(w_{i,j}) \\
  \rho_Y(x_{i,n,j}) &=& \pi_X(w_{i,n})a_j \\
  \rho_Y(y) &=& \pi_X(w_{n,n})^*+\pi_X(w_{n,n+1})^* \\
  \rho_Y(y') &=& \pi_X(w_{n,n+1})^* \\
  \rho_Y(y_i) &=& a_i(\pi_X(w_{n,n})^*+\pi_X(w_{n,n+1})^*)
\end{eqnarray*}
and a $*$-homomorphism $\rho_B:B\to \Oo_X$ satisfying
\[
 \begin{array}{ll}
  \rho_B(R_i) = P_i & \rho_B(Q_i) = a_i
 \end{array}
\]
where $a_i$ is the projection
\[
 a_i = (\pi_X(w_{n,n})+\pi_X(w_{n,n+1}))^i \pi_A(P_{n+1})(\pi_X(w_{n,n})^*+\pi_X(w_{n,n+1})^*)^i \in \Oo_X.
\]
We can easily show from Lemma \ref{lemma:idealinB} that $(\rho_Y,\rho_B):(Y,B)\to \Oo_X$ is also a covariant representation.

Now, each representation induces an injective map, $\Pi_X:\Oo_X\to \Oo_Y$ and $\Pi_Y:\Oo_Y\to \Oo_X$ so all that remains to be seen is that $\Pi_X\circ\Pi_Y = 1_{\Oo_Y}$ and $\Pi_Y\circ\Pi_X=1_{\Oo_X}$. This is straightforward, and the result follows.
\end{proof}

Now let $\omega_Y:Y\to Z$ be the unique linear map satisfying $\pi_Z(\omega_Y(\xi)) = \Psi(\Pi_Y(\pi_Y(\xi)))$ and $\omega_B:B\to C$ be the unique homomorphism satisfying $\pi_C(\omega_B(b)) = \Psi(\Pi_Y(\pi_B(b)))$. Then it is easily checked that $(\omega_Y,\omega_B)$ is a $C^*$-correspondence morphism, and furthermore $\Omega = \beta\circ\Psi\circ\Pi_Y$.

Before we can apply Theorem \ref{thm:pullback} we need to show that these $C^*$-correspondences satisfy the hypotheses. We have both $\psi_X$ and $\omega_Y$ surjective, and also $\psi_A$ and $\omega_B$ surjective. Lemma \ref{lemma:idealinA} and Lemma \ref{lemma:idealinB} imply that both left actions are by compact operators.  Furthermore, $\psi_A(\ker(\phi_X)) = \{0\} = \omega_B(\ker(\phi_Y))$. If we set $J_A := C^*(P_1,\dots P_n)$, then it is clear that $J_A$ is an ideal in $A$ and $A = J_A \oplus \ker(\phi_X)$, so $\ker(\phi_X)$ is complemented. The ideal $\ker(\phi_Y)$ is just the zero ideal, so it is also trivially complemented.
Hence we can use Theorem \ref{thm:pullback}. There is an isomorphism between the pullback $\Oo_X \oplus_{\Oo_Z} \Oo_Y$ and the $C^*$-algebra $\Oo_{X\dsz Y}$, where the underlying Banach space satisfies
\begin{align*}
 X\dsz Y = \clsp&(\{(w_{i,j},x_{i,j}:1\leq i \leq n-1, i\leq j \leq n+1  \\  &\cup \{(0,x_{i,n,j}):1\leq i \leq n, j\geq 1\} \cup \{(0,y_i):i\geq 1\} \\ &\cup \{(w_{n,n},y), (w_{n,n+1}, 0), (0, y')\})
\end{align*}
and the pullback $C^*$-algebra $A\dsc B$ is generated by projections
\[
 \{(P_i,R_i:1\leq i \leq n)\} \cup \{(0,Q_i:i\geq 1\} \cup \{(P_n,R_n),(P_{n+1},0),(0,R_{n+1})\}
\]
where $(0,Q_i) \leq (P_n,R_n)$, and all other projections are orthogonal. So we have shown that $C(S^{2n}_{q,\beta}) \cong \Oo_{X\dsz Y}$
is a $C^*$-algebra associated to a $C^*$-correspondence.

\section{$C^*$-algebras of Labelled Graphs} \label{sec:labelledgraphs}

Now, we can use this characterisation to prove that $C(S^{2n}_{q,\beta})$ is actually a $C^*$-algebra associated to a labelled graph, first introduced by Bates and Pask in \cite{BaPa}. We begin by recalling basic definitions and results from \cite{BaPa}, and then showing that with an appropriate notion of morphisms between labelled spaces, the class of labelled spaces forms a category. Furthermore, there is a functor from this category to the category of $C^*$-correspondences constructed earlier.

\begin{defn} A labelled graph $(E,\Ll)$ over an alphabet $\Aa$ is a directed graph $E$ together with a labelling map $\Ll: E^1\to \Aa$ which assigns to each edge $e\in E^1$ a label $a\in\Aa$.
\end{defn}

It is worth noting at this point that there are no conditions imposed on the map $\Ll$; it is not assumed to be either injective or surjective. However, in practice we generally make $\Ll$ surjective by simply discarding any elements of $\Aa$ which do not lie in the range of $\Ll$.

Define a \emph{word} to be a finite string $a_1a_2\dots a_n$, with each $a_i\in \Aa$. We write $\Aa^*$ for the collection of all words in $\Aa$ and then for any $n\in \NN$, the labelling can be extended to $E^n$ by defining $\Ll(e_1e_2\dots e_n) = \Ll(e_1)\Ll(e_2)\dots\Ll(e_n) \in \Aa^*$. We write $\Ll^*(E) = \cup_{n\geq 1} \Ll(E^n)$. An element $\alpha\in\Ll^*(E)$ is called a \emph{labelled path}.

\begin{defn}
Let $(E,\Ll)$ be a labelled graph. We say that $(E,\Ll)$ is \emph{left-resolving} if for all $v\in E^0$, the labelling $\Ll$ restricted to $r^{-1}(v)$ is injective.
\end{defn}
In other words, a labelled graph is left-resolving if all edges entering a particular vertex carry different labels.

\begin{defn}
Let $(E,\Ll)$ be a labelled graph, let $A\subset E^0$ and let $\alpha\in\Ll(E^*)$ be a labelled path. The \emph{relative range of $\alpha$ in A}, denoted $r(A,\alpha)$ is defined to be the set
\[
 r(A,\alpha) := \{r(\lambda):\lambda\in E^*, \Ll(\lambda) = \alpha \mbox{ and } s(\lambda)\in A\}
\]
\end{defn}
Now, let $\Bb\subset 2^{E^0}$ be a collection of subsets of $E^0$. We say that $\Bb$ is \emph{closed under relative ranges} if for any $A\in\Bb$ and $\alpha\in \Ll^*(E)$, we have $r(A,\alpha)\in \Bb$. We say that $\Bb$ is \emph{accommodating} for $(E,\Ll)$ if it is closed under relative ranges, contains $r(\alpha)$ for all $\alpha\in\Ll^*(E)$, contains $\{v\}$ whenever $v$ is a sink, and is also closed under finite unions and intersections.

\begin{defn} A labelled space is a triple $(E,\Ll,\Bb)$ where $(E,\Ll)$ is a labelled graph and $\Bb$ is accommodating for $(E,\Ll)$.
\end{defn}

\begin{defn} Let $(E,\Ll,\Bb)$ be a labelled space. We say $(E,\Ll,\Bb)$ is \emph{left-resolving} if $(E,\Ll)$ is left resolving and we say that $(E,\Ll,\Bb)$ is \emph{weakly left-resolving} if for every $A,B\in\Bb$ and $\alpha\in\Ll^*(E)$, we have
\[
 r(A,\alpha) \cap r(B,\alpha) = r(A\cap B,\alpha)
\]
\end{defn}

Note that we will need to assume our labelled spaces are weakly left-resolving in order for the associated $C^*$-algebras to be non-degenerate.

We now have the required definitions to define the $C^*$-algebra associated to a labelled space.

\begin{defn} \label{defn:labelledgraphrepn}
Let $(E,\Ll,\Bb)$ be a weakly left-resolving labelled space. A representation of $(E,\Ll,\Bb)$ is a collection $\{p_A:A\in \Bb\}$ of projections and a collection $\{s_a:a\in\Ll(E^1)\}$ of partial isometries such that:
\begin{enumerate}
 \item For $A,B\in\Bb$, we have $p_Ap_B = p_{A\cap B}$ and $p_{A\cup B} = p_A + p_B - p_{A\cap B}$ where $p_\emptyset = 0$;
 \item For $a\in\Ll(E^1)$ and $A\in\Bb$, we have $p_A s_a = s_a p_{r(A,a)}$;
 \item For $a,b\in\Ll(E^1)$, we have $s_a^*s_a = p_{r(a)}$ and $s_a^*s_b = 0$ unless $a=b$; and
 \item For $A\in\Bb$ define $L^1(A) := \{a\in\Ll(E^1):s(a)\cap A \neq \emptyset\}$. Then if $L^1(A)$ is finite and non-empty, we have
  \[
   p_A = \sum_{a\in L^1(A)} s_a p_{r(A,a)} s_a^* + \sum_{v\in A:v\mbox{ \tiny{is a sink}}} p_{\{v\}}.
  \]
\end{enumerate}
\end{defn}

As noted in Remark 3.2 of \cite{BaPa2}, there was an error in the original definition of $C^*(E,\Ll,\Bb)$, where projections at sinks would be degenerate. Hence the proof of the existence of the $C^*$-algebra associated to a labelled graph given in \cite{BaPa} doesn't hold when the underlying graph contains sinks. Since we will be looking specifically at a labelled graph with sinks, we reprove the result in the required generality here. The proof closely follows that of the proof of the existence of $C^*(E)$ when $E$ is a directed graph with sinks.

\begin{prop}
Let $(E,\Ll, \Bb)$ be a weakly left-resolving labelled space. Then there exists a $C^*$-algebra $B$ generated by a universal representation of $\{s_a, p_A\}$ of $(E,\Ll, \Bb)$. Furthermore, the $s_a$'s are nonzero and every $p_A$ with $A\neq \emptyset$ is nonzero.
\end{prop}

\begin{proof}
Fix a weakly left-resolving labelled space $(E,\Ll, \Bb)$. We define a new directed graph $F$ by
\[
 F^0 = E^0 \cup \{v_i:v \mbox{ is a sink}, i\in\NN\} \ \mbox{ and } \ F^1 = E^1 \cup \{f_{v,j}:v \mbox{ is a sink}, j\in\NN\}
\]
and extend the range and source maps on $E$ to $F$ by
\[
 s(f_{v,i}) = \left\{ \begin{array}{l} v \mbox{ if } i=1 \\ v_{i-1} \mbox{ if } i>1 \end{array} \right. \ \mbox{ and } \ r(f_{v,i}) = v_i.
\]

We also extend the labelling $\Ll$ on $E$ to a labelling on $F$ by $\Ll|_{(F^1\backslash E^1} = \id$. Define the set $\Bb_F \subset 2^{F^0}$ to be the smallest subset containing $\Bb \cup \{\{v_i\}:v_i\in F^0\backslash E^0\}$ that is accommodating for $(F,\Ll)$. Then it is easily checked that $(F,\Ll,\Bb_F)$ is a labelled space, and weakly left resolving if and only if $(E,\Ll,\Bb)$ is. Furthermore, since $F$ has no sinks it follows from \cite[Theorem 4.5]{BaPa} that $C^*(F,\Ll,\Bb_F)$ exists.

Let $\{s_a:a\in \Ll(F^1)\}, \{p_A:A\in \Bb_F\}$ be the universal representation of $(F,\Ll,\Bb_F)$. Then we claim that the restriction $\{s_a:a\in\Ll(E^1)\}, \{p_A:A\in \Bb\}$ is a representation of $(F,\Ll,\Bb_F)$. Indeed, condition (1) - (3) hold in $(E,\Ll,\Bb)$ as they do in $(F,\Ll,\Bb_F)$. However condition (4) is not so clear as any sink in $E$ is not a sink in $F$. So fix a set $A_E\in \Bb$ satisfying the conditions of (4). If $A_E$ does not contain any sinks then the result is obvious, so assume that $A_E$ contains finitely many, but at least one sink.

Denote by $A_F$ the identical set considered as an element of $\Bb_F$. Then it is easy to see that $L^1_{A_E} \cup \{\{f_{v,1}\}:v\in A_E \mbox{ is a sink}\} = L^1_{A_F}$. So
\begin{eqnarray*}
 p_{A_E} &=& p_{A_F} \\
 &=& \sum_{a\in L^1_{A_F}} p_{A_F}s_a s_a^* \\
 &=& \sum_{a\in L^1_{A_E}} p_{A_E}s_a s_a^* + \sum_{\{\{f_{v,1}\}:v\in A_E \mbox{\tiny{ is a sink}}\}} p_{A_E} s_{f_{v,1}} s_{f_{v,1}}^* \\
 &=& \sum_{a\in L^1_{A_E}} p_{A_E}s_a s_a^* + \sum_{\{\{f_{v,1}\}:v\in A_E \mbox{\tiny{ is a sink}}\}} p_{\{v\}} s_{f_{v,1}} s_{f_{v,1}}^* \\
 &=& \sum_{a\in L^1_{A_E}} p_{A_E}s_a s_a^* + \sum_{\{v\in A_E \mbox{\tiny{ is a sink}}\}} \left(\sum_{b\in L^1_{\{v\}}} p_{\{v\}} s_b s_b^* \right) \mbox{ since } L^1_{\{v\}}= \{f_{v,1}\} \\
 &=& \sum_{a\in L^1_{A_E}} p_{A_E}s_a s_a^* + \sum_{\{v\in A_E \mbox{\tiny{ is a sink}}\}} p_{\{v\}}  \ \ \mbox{ since $\{v\}\in \Bb_F$ satisfies (4)}
\end{eqnarray*}
as required. So we have the required result.
\end{proof}

We can now make the following definition.

\begin{defn}
Let $(E,\Ll,\Bb)$ be a weakly left-resolving labelled space. Define $C^*(E,\Ll,\Bb)$ to be the universal $C^*$-algebra generated by representations of $(E,\Ll,\Bb)$.
\end{defn}

Now that we have the definitions in place, we are ready to prove that the labelled spaces form a category and that there is a functor from this category to the category of $C^*$-correspondences such that the associated $C^*$-algebras are isomorphic.

We define the objects of our category to be labelled spaces $(E,\Ll,\Bb)$. A morphism $\psi_E: (E,\Ll_E,\Bb_E) \to (F,\Ll_F,\Bb_F)$ is a pair of maps
\[
 \psi_{E^0} : E^0 \cup \{0\} \to F^0 \cup \{0\} \ \mbox{ and } \ \psi_{E^1}:E^1 \cup \{0\} \to F^1 \cup \{0\}
\]
satisfying
\begin{enumerate}
 \item[(L1)] \label{L1} $\psi_{E^0}(0) = \psi_{E^1}(0) = 0$,
 \item[(L2)] \label{L2} $\psi_{E^0}(u) = \psi_{E^0}(v) \neq 0$ implies $v=w$,
 \item[(L3)] \label{L3} $\Ll_F(\psi_{E^1}(e)) = \Ll_F(\psi_{E^1}(f)) \neq 0$ implies $\Ll_E(e) = \Ll_E(f)$ where $\Ll_E(0) = \Ll_F(0) = 0$,
 \item[(L4)] \label{L4} $\{\psi_{E^0}(v):v\in r(a)\} = \{r(\psi_{E^1}(e)):\Ll_E(e) = a\}$ for all $a \in \Ll_E(E^1)$,
 \item[(L5)] \label{L5} $\{\psi_{E^0}(v):v\in A\} \setminus \{0\} \in \Bb_F$ for all $A\in \Bb_E$ and $0 < |L^1(A)| < \infty$ implies $0 < |L^1(\psi_{E^0}(A))| < \infty$.
\end{enumerate}
Properties (L3) and (L5) ensure that we can extend these maps to $\Ll_E(E^1)$ and $\Bb_E$ so we define $\psi_E(a) := \Ll_F(\psi_{E^1}(e))$ where $\Ll_E(e) = a$ and $\psi_E(A) := \{\psi_{E^0}(v):v\in A\} \setminus \{0\}$ for all $a\in \Ll(E^1)$ and $A\in\Bb_E$.

It is not hard to see that this gives a category with the obvious domain and range maps, identity morphism and composition of morphisms.

Now that we want to show that $C^*(E,\Ll,\Bb)$ can be naturally realised as a $C^*$-algebra associated to a $C^*$-correspondence, and that there is a functor between the category of labelled spaces and the category of $C^*$-correspondences that preserves the associated $C^*$-algebras up to isomorphism.

\begin{prop} \label{prop:labelledgraphcorrespondence}
Let $(E,\Ll,\Bb)$ be a weakly left-resolving labelled space. Then there exists a $C^*$-correspondence $(X(E),A(E))$ such that $\Oo_{X(E)} \cong C^*(E,\Ll,\Bb)$.
\end{prop}

\begin{proof}
Let $A(E)$ be the $C^*$-subalgebra of $C^*(E,\Ll,\Bb)$ generated by the set of projections $\{p_A:A\in\Bb\}$ and $X(E)$ be the Banach subspace of $C^*(E,\Ll,\Bb)$ spanned by the elements $\{s_a p_B: a\in\Ll(E^1), B\in \Bb\}$. Define the right action of $A(E)$ on $X(E)$ simply by multiplication, and the inner-product by'

\[
 \langle s_a p_B, s_c p_D\rangle  := (s_a p_B)^* s_c p_D = \delta_{a,c} \ p_{B\cap D\cap r(a)}.
\]

Similarly, define the left action $\phi_{X(E)}$ by
\[
 \phi_{X(E)}(p_C) s_a p_B := s_a p_{B\cap r(a,C)}.
\]
Then we have a $C^*$-correspondence. In order to see that there is an isomorphism between $\Oo_{X(E)}$ and $C^*(E,\Ll,\Bb)$, we show that there is a covariant representation of $(E,\Ll,\Bb)$ inside $\Oo_{X(E)}$. Indeed, by construction of $X(E)$ and $A(E)$ it is easily show that
\[
 \{\pi_{A(E)}(p_A):A\in\Bb\}, \{\pi_{X(E)}(s_a p_{r(a)}):a\in\Ll(E^1)\}
\]
is such a representation. Then since the universal covariant representation $(\pi_{X(E)},\pi_{A(E)})$ of $(X(E),A(E))$ admits a gauge action, the gauge invariant uniqueness theorem for labelled graphs, \cite[Theorem 5.3]{BaPa} implies that $C^*(E,\Ll,\Bb)$ is isomorphic to the $C^*$-subalgebra of $\Oo_{X(E)}$ generated by $\{\pi_{X(E)}(p_A):A\in\Bb\}$ and $\{t_{X(E)}(s_a p_{r(a)}):a\in\Ll(E^1)\}$. Also, since we have $\pi_{X(E)}(s_a p_A) = \pi_{X(E)}(s_a p_{r(a)}) \pi_{A(E)}(p_A)$, this sub-algebra contains the generators of $\Oo_{X(E)}$, so we have the required isomorphism.
\end{proof}

\begin{prop}
There is a functor between the category of labelled spaces and the category of $C^*$-correspondences constructed earlier such that $(E,\Ll,\Bb) \mapsto (X(E),A(E))$.
\end{prop}

\begin{proof}
Fix a morphism $\psi_E$ between two labelled spaces $(E,\Ll_E,\Bb_E)$ and $(F,\Ll_F,\Bb_F)$. We begin by defining a map $G$ on object sets by $G:(E,\Ll,\Bb) \mapsto (X(E),A(E))$. We extend this map to the set of morphisms by defining $G:\psi_E \mapsto (\psi_{X(E)},\psi_{A(E)})$ where $\psi_{X(E)}:X(E) \to X(F)$ and $\psi_{A(E)}:A(E) \to A(F)$ are a pair of maps satisfying
\[
 \psi_{X(E)}(s_a p_B) = s_{\psi_E(a)} p_{\psi_E(B)} \ \mbox{ and } \ \psi_{A(E)}(p_B) = p_{\psi_E(B)}
\]
for all $a\in\Ll_E(E^1)$ and $B\in\Bb_E$. In order to show that $G$ is in fact a functor, we need to show that $G(\psi_E) = (\psi_{X(E)},\psi_{A(E)})$ is a morphism of $C^*$-correspondences. For (C1), fix $s_a p_B$ and $s_c p_D \in X(E)$. It is clear from (L2) that we have $\psi_E(A\cap C) = \psi_E(A) \cap \psi_E(C)$ and from (L3) that $\psi_E(a) = \psi_E(c)$ if and only if $a=c$. Then
\begin{eqnarray*}
 \langle \psi_{X(E)}(s_a p_B), \psi_{X(E)}(s_c p_D)\rangle &=& \delta_{\psi_E(a),\psi_E(c)} p_{\psi_E(B) \cap \psi_E(D) \cap r(\psi_E(a))} \\
 &=& \delta_{\psi_E(a),\psi_E(c)} p_{\psi_E(B) \cap \psi_E(D) \cap \psi_E(r(a))} \mbox{ from (L4)} \\
 &=& \delta_{a,c} \ p_{\psi_E(B\cap D\cap r(a))} \\
 &=& \psi_{X(E)}(\delta_{a,c} \ p_{B\cap D\cap r(a)}) \\
 &=& \psi_{X(E)}(\langle s_a p_B, s_c p_D\rangle)
\end{eqnarray*}
as required. A similar calculation also shows that (C2) holds. Now, note that condition (4) of Definition \ref{defn:labelledgraphrepn} implies that
\[
 J_{X(E)} = \mbox{ideal generated by } \{p_A: L^1_A \mbox{ is finite and non-empty}\}.
\]
So (C3) easily follows from (L5). The other properties of the functor are easily shown to hold by definition of the map. Finally, that this functor preserves the associated $C^*$-algebras follows directly from Proposition \ref{prop:labelledgraphcorrespondence}.
\end{proof}

As an example of a labelled graph $C^*$-algebra, we show that for any $n\in\NN$, there exists a labelled space $(E_n,\Ll,\Bb)$ such that $C^*(E_n,\Ll,\Bb) \cong C(S^{2n}_{q,\beta})$.
\begin{exm}
Fix $n\geq 1$, and define a directed graph $E_n$ by
\[
 E_n^0 = \{u_i:1\leq i \leq n-1\} \cup \{v_i:i\geq 1\} \cup \{w_1,w_2\}
\]
and
\[
 E_n^1 = \{e_{i,j}:1\leq i,j \leq n-1\} \cup \{f_{i,j}:1\leq i \leq n-1, j\geq 1\} \cup \{g_i,h_i:i\geq 1\}
\]
with range and source maps which satisfy
\[
 \begin{array}{ll}
  s(e_{i,j}) = u_i & r(e_{i,j}) = u_j \\
  s(f_{i,j}) = u_i & r(f_{i,j}) = \left\{
  \begin{array}{l}
   w_1 \mbox{ for } j=1 \\
   w_2 \mbox{ for } j=2 \\
   v_{j-2} \mbox{ for } j\geq 3
  \end{array} \right. \\
  s(g_i) = \left\{
  \begin{array}{l}
  w_2 \mbox{ for } i=1 \\
  v_{i-1} \mbox{ for } i\geq 2
  \end{array} \right. & r(g_i) = v_i \\
  s(h_i) = v_i & r(h_i) = w_1
 \end{array}
\]
Now, define an alphabet $\mathcal{A}$ by
\[
 \mathcal{A} := \{e_{i_j}:1\leq i,j\leq n-1\} \cup \{f_i:1\leq i \leq n-1\} \cup \{g,h\}
\]
and a labelling $\Ll:E_n^1 \to \mathcal{A}$ which satisfies
\[
 \begin{array}{ll}
  \Ll(e_{i,j}) = e_{i,j} & \Ll(f_{i,j}) = f_i \\
  \Ll(g_i) = g & \Ll(h_i) = h
 \end{array}
\]
for all $i,j\in\NN$.
Then $(E_n,\Ll)$ is a labelled graph. For example, the labelled graph $(E_5,\Ll)$ looks like
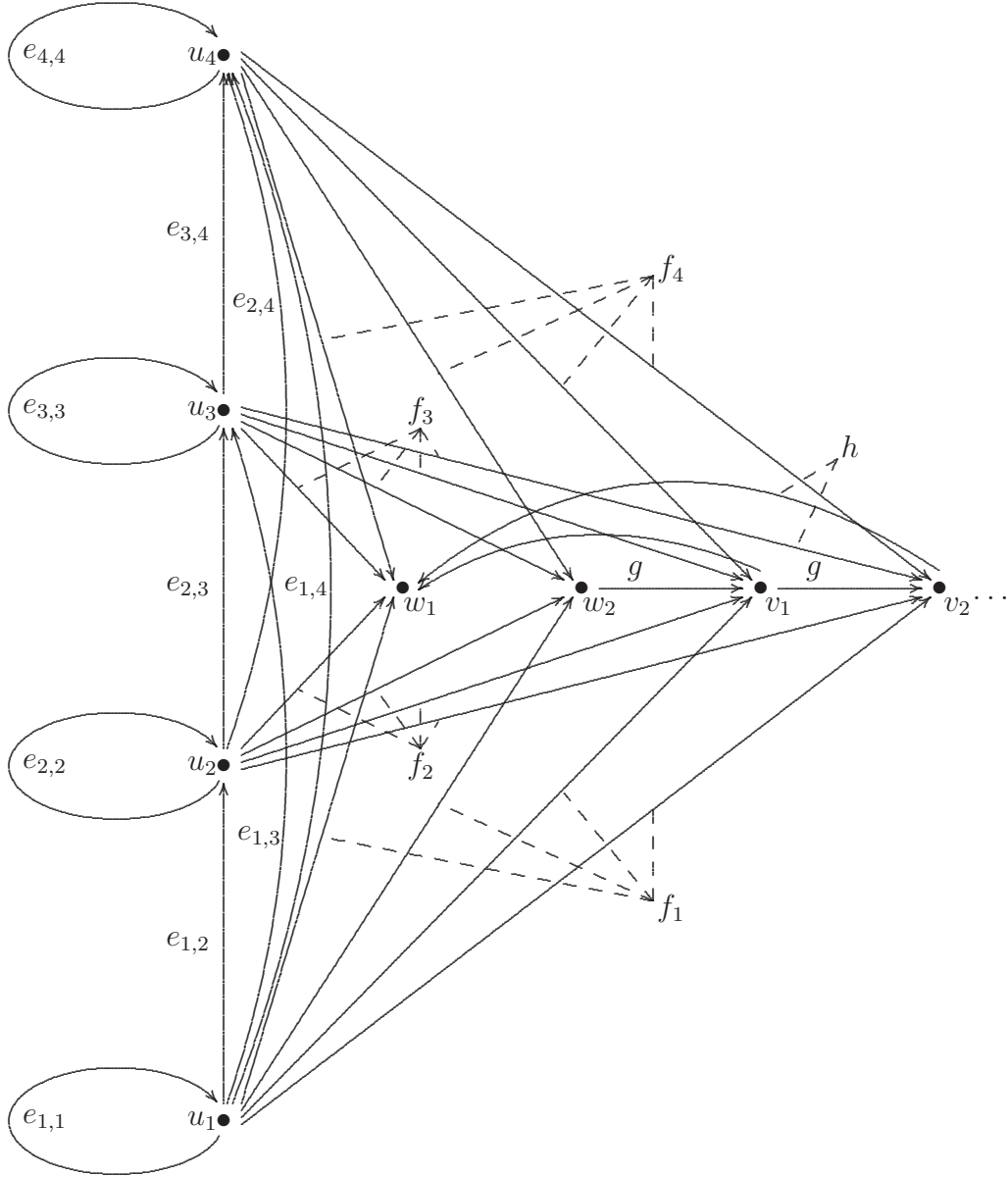
\[
 \beginpicture
  \setcoordinatesystem units <2.4cm,2.4cm>
  \put{$ \ $} at 0 3.5
  \put{$ \ $} at 0 -3.5
  \put{$ \ $} at -1 0
  \put{$\bullet$} at 0 -1
  \put{$\bullet$} at 0 1
  \put{$\bullet$} at 0 -3
  \put{$\bullet$} at 0 3
  \put{$\bullet$} at 1 0
  \put{$\bullet$} at 2 0
  \put{$\bullet$} at 3 0
  \put{$\bullet$} at 4 0
  \put{$e_{1,1}$} at -1 -3
  \put{$e_{2,2}$} at -1 -1
  \put{$e_{3,3}$} at -1 1
  \put{$e_{4,4}$} at -1 3
  \put{$e_{1,2}$} at -.2 -2
  \put{$e_{2,3}$} at -.2 0
  \put{$e_{3,4}$} at -.2 2
  \put{$e_{1,3}$} at .2 -1.4
  \put{$e_{1,4}$} at .46 0
  \put{$e_{2,4}$} at .17 1.6
  \put{$\dots$} at 4.3 -.05
  \put{$u_1$} at -.12 -3
  \put{$u_2$} at -.12 -1
  \put{$u_3$} at -.12 1
  \put{$u_4$} at -.12 3
  \put{$w_1$} at 1.1 -.1
  \put{$w_2$} at 2.1 -.1
  \put{$v_1$} at 3.1 -.1
  \put{$v_2$} at 4.1 -.1
  \put{$h$} at 3.5 .8
  \put{$f_4$} at 2.5 1.8
  \put{$f_3$} at 1.1 1
  \put{$f_2$} at 1.1 -1
  \put{$f_1$} at 2.5 -1.8
  \put{$g$} at 2.3 0.1
  \put{$g$} at 3.3 0.1
  \setdashes
  \setlinear
  \plot 3.43 .73 3.1 .53 /
  \plot 3.43 .73 3.2 .23 /
  \plot 2.4 1.76 0.55 1.4 /
  \plot 2.4 1.76 1.2 1.21 /
  \plot 2.4 1.76 1.9 1.15 /
  \plot 2.4 1.76 2.4 1.25 /
  \plot 1.1 .9 0.42 0.57 /
  \plot 1.1 .9 .85 .57 /
  \plot 1.1 .9 1.1 .65 /
  \plot 1.1 .9 1.2 .75 /
  \plot 2.4 -1.76 0.55 -1.4 /
  \plot 2.4 -1.76 1.2 -1.21 /
  \plot 2.4 -1.76 1.9 -1.15 /
  \plot 2.4 -1.76 2.4 -1.25 /
  \plot 1.1 -.9 0.42 -0.57 /
  \plot 1.1 -.9 .85 -.57 /
  \plot 1.1 -.9 1.1 -.65 /
  \plot 1.1 -.9 1.2 -.75 /
  \setsolid
  \arrow <0.15cm> [0.25,0.75] from 0 -.9 to 0 .9
  \arrow <0.15cm> [0.25,0.75] from 0 -2.9 to 0 -1.1
  \arrow <0.15cm> [0.25,0.75] from 0 1.1 to 0 2.9
  \arrow <0.15cm> [0.25,0.75] from .1 -.9 to .9 -.05
  \arrow <0.15cm> [0.25,0.75] from .1 -.94 to 1.9 -.05
  \arrow <0.15cm> [0.25,0.75] from .1 -.98 to 2.9 -.05
  \arrow <0.15cm> [0.25,0.75] from .1 -1.02 to 3.9 -.05
  \arrow <0.15cm> [0.25,0.75] from .1 .9 to .9 .05
  \arrow <0.15cm> [0.25,0.75] from .1 .94 to 1.9 .05
  \arrow <0.15cm> [0.25,0.75] from .1 .98 to 2.9 .05
  \arrow <0.15cm> [0.25,0.75] from .1 1.02 to 3.9 .05
  \arrow <0.15cm> [0.25,0.75] from .1 -2.9 to .95 -.05
  \arrow <0.15cm> [0.25,0.75] from .1 -2.94 to 1.95 -.05
  \arrow <0.15cm> [0.25,0.75] from .1 -2.98 to 2.95 -.05
  \arrow <0.15cm> [0.25,0.75] from .1 -3.02 to 3.95 -.05
  \arrow <0.15cm> [0.25,0.75] from .1 2.9 to .95 .05
  \arrow <0.15cm> [0.25,0.75] from .1 2.94 to 1.95 .05
  \arrow <0.15cm> [0.25,0.75] from .1 2.98 to 2.95 .05
  \arrow <0.15cm> [0.25,0.75] from .1 3.02 to 3.95 .05
  \arrow <0.15cm> [0.25,0.75] from 2.1 0 to 2.9 0
  \arrow <0.15cm> [0.25,0.75] from -0.05 -2.88 to -0.041 -2.89
  \arrow <0.15cm> [0.25,0.75] from -0.05 -0.88 to -0.041 -0.89
  \arrow <0.15cm> [0.25,0.75] from -0.05 1.12 to -0.041 1.11
  \arrow <0.15cm> [0.25,0.75] from -0.05 3.12 to -0.041 3.11
  \arrow <0.15cm> [0.25,0.75] from 3.1 0 to 3.9 0
  \arrow <0.15cm> [0.25,0.75] from .06 0.86 to .05 0.9
  \arrow <0.15cm> [0.25,0.75] from .06 2.86 to .05 2.9
  \arrow <0.15cm> [0.25,0.75] from .03 2.86 to .03 2.9
  \arrow <0.15cm> [0.25,0.75] from 1.1 0 to 1.09 -.01
  \arrow <0.15cm> [0.25,0.75] from 1.1 0.04 to 1.09 .03
  \ellipticalarc axes ratio 2:1 320 degrees from -0.05 -2.88 center at -.6 -3
  \ellipticalarc axes ratio 2:1 320 degrees from -0.05 -0.88 center at -.6 -1
  \ellipticalarc axes ratio 2:1 320 degrees from -0.05  1.12 center at -.6  1
  \ellipticalarc axes ratio 2:1 320 degrees from -0.05  3.12 center at -.6  3
  \setquadratic
  \plot 1.1 0 2 .3 3 .1 /
  \plot 1.1 .04 2.5 .6 4 .1 /
  \plot .02 -2.9 .35 -1 .05 0.9 /
  \plot .02 -.9 .35 1 .02 2.9 /
  \plot .05 -2.9 .6 0 .05 2.9 /
 \endpicture
\]
where the dashed lines indicate that the edges all carry the same label.

Note that the vertex $w_1$ is a sink for all $E_n$. Now, it is obvious that $(E_n,\Ll)$ is not left-resolving for any $n\geq 2$, since the vertex $u_i$ emits infinitely many edges labelled $f_i$. So in order to associate a $C^*$-algebra to this labelled graph we must find a collection $\Bb \subset 2^{E_n^0}$ such that $(E_n,\Ll,\Bb)$ is a weakly left-resolving labelled space.
Let
\[
 B = \{\{u_i\}:1\leq i\leq n-1\} \cup \{\{w_1\}\} \cup \{A_i:i\geq 1\} \cup \{A_1\cup\{w_2\}\}
\]
where $A_i = \{v_j:j\geq i\}$. Then let $\Bb$ be the closure of $B$ under finite unions and intersections. It is easy to see that $\Bb$ contains $r(\alpha)$ for all $\alpha\in\Ll(E^*)$, and is closed under relative ranges.
\end{exm}

\begin{theorem} The $C^*$-algebra $C^*(E_n,\Ll,\Bb)$ is isomorphic to the $C^*$-algebra $C(S_{q,\beta}^{2n})$.
\end{theorem}

\begin{proof} We prove this by showing that there exists a morphism of $C^*$-correspondences $(\rho_{X\dsz Y},\rho_{A\dsc B}):(X\dsz Y, A\dsc B) \to (X(E_n),A(E_n))$ that induces a $C^*$-algebra isomorhism $P:\Oo_{X\dsz Y} \to \Oo_{X(E_n)} \cong C^*(E_n,\Ll,\Bb)$. First, define the linear map $\rho_{X\dsz Y}:X\dsz Y \to C^*(E_n,\Ll,\Bb)$ by
\[
 \begin{array}{l}
  \rho_{X\dsz Y}(w_{i,j},x_{i,j}) = S_{e_{i,j}}P_{\{u_j\}}, \ 1\leq i,j \leq n-1 \\
  \rho_{X\dsz Y}(w_{i,n},x_{i,n}) = S_{f_i} P_{A_1}, \ 1\leq i \leq n-1 \\
  \rho_{X\dsz Y}(w_{i,n+1},x_{i,n+1}) = S_{f_i}(P_{\{w_1\}}+(P_{\{w_2\}\cup A_1} - P_{A_1})), \ 1\leq i \leq n-1 \\
  \rho_{X\dsz Y}(w_{n,n},y) = S_g P_{A_1} \\
  \rho_{X\dsz Y}(w_{n,n+1},0) = S_h P_{\{w_1\}} \\
  \rho_{X\dsz Y}(0,x_{i,n,j}) = S_{f_i}(P_{A_j}-P_{A_{j+1}}), \ 1\leq i \leq n-1, j \geq 1 \\
  \rho_{X\dsz Y}(0,y') = S_g(P_{A_1}-P_{A_2}) \\
  \rho_{X\dsz Y}(0,y_i) = S_g(P_{A_{i+1}}-P_{A_{i+2}}), i \geq 1
 \end{array}
\]
and a map $\rho_{A\dsc B}:A\dsc B\to A(E_n)$ by
\[
 \begin{array}{ll}
  \rho_{A\dsc B}(P_i,R_i) = P_{\{u_i\}} & \rho_{A\dsc B}(P_n,R_n) = P_{A_1} \\
  \rho_{A\dsc B}(P_{n+1},0) = P_{\{w_1\}} & \rho_{A\dsc B}(0,R_{n+1}) = P_{\{w_2\}\cup A_1} - P_{A_1} \\
  \rho_{A\dsc B}(0,Q_i) = P_{A_i}-P_{A_{i+1}}. &
 \end{array}
\]
By definition, $\rho_{X\dsz Y}$ is a linear map and it is easy to see that $\rho_{A\dsc B}$ is an injective homomorphism. Furthermore, routine calculations show that this pair $(\rho_{A\dsc B},\rho_{X\dsz Y})$ satisfies conditions (C1) and (C2). For (C4), we must find the ideal $J_{X\dsz Y}$. We know that $J_X\subset A$ is the ideal generated by the projections $P_i$ where $1\leq i\leq n$ and the ideal $J_Y=B$, so it is not hard to see that $J_{X\dsz Y}$ is the ideal generated by the projections $(P_i,R_i)$ where $1\leq i\leq n, (0,R_{n+1})$ and $(0,Q_j)$ for $j\geq 1$. Now by noticing that $L^1_A$ is finite for all $A\in \Bb$ and $\{w_1\} \in \Bb$ is the only subset satisfying $L^1_{\{w_1\}} = \emptyset$, we see that (C3) is satisfied. Then it follows easily from condition (4) of Definition \ref{defn:labelledgraphrepn} that $(\rho_{A\dsc B},\rho_{X\dsz Y})$ satisfies (C4).

So we have a $C^*$-correspondence morphism, now all that remains to be seen is that this $C^*$-correspondence morphism induces an isomorphism $\Oo_{X\dsz Y} \to \Oo_{X(E_n)}$. To do this we use the gauge invariant uniqueness theorem. Firstly, if we compose our morphism with the universal covariant representation $(\pi_{X(E_n)}, \pi_{A(E_n)})$ of $(X(E_n),A(E_n))$ we get a covariant representation $(\pi_{X(E_n)}\circ\rho_{X\dsz Y}, \pi_{A(E_n)}\circ \rho_{A\dsc B})$ of $(X\dsz Y,A\dsc B)$ on $\Oo_{X(E_n)}$. We know from Section 5 in \cite{BaPa} that there is a gauge action on $C^*(E_n,\Ll,\Bb)$ and it is clear from the definitons that this also defines a gauge action on $C^*(\pi_{X(E_n)}\circ\rho_{X\dsz Y}, \pi_{A(E_n)}\circ \rho_{A\dsc B})$. Furthermore, it is not hard to see that $\rho_{A\dsc B}$ is injective and hence Theorem \ref{thm:gaugeinvariantuniquenesstheorem} implies that $P:\Oo_{X\dsz Y} \to C^*(\pi_{X(E_n)}\circ\rho_{X\dsz Y}, \pi_{A(E_n)}\circ \rho_{A\dsc B})$ is an isomorphism. Finally, the image of the covariant representation $(\pi_{X(E_n)}\circ\rho_{X\dsz Y},\pi_{A(E_n)}\circ\rho_{A\dsc B})$ contains all the generators of $\Oo_{X(E_n)}$, so $C^*(\pi_{X(E_n)}\circ\rho_{X\dsz Y}, \pi_{A(E_n)}\circ \rho_{A\dsc B}) = \Oo_{X(E_n)}$ and we have the required isomorphism.
\end{proof}

As a final result, we prove that - at least in the 2 dimensional case - the mirror quantum sphere cannot be realised as a graph algebra. This result was originally stated in \cite{HaMaSz}, though there was no proof given. We present the full proof here to fill this gap in the literature.

\begin{theorem}
There is no directed graph $E$ such that $C^*(E) \cong C(S^2_{q,\beta})$.
\end{theorem}

\begin{proof}
Suppose, by way of contradiction, that $E$ is a directed graph such that $C^*(E) \cong C(S^2_{q,\beta})$. 
Since $C^*(E)$ is unital, graph $E$ has only finitely many vertices, but a priori it is possible that $E$ has 
infinitely many edges. There exists an exact sequence 
$$ 0 \longrightarrow \Kk\oplus\Kk \longrightarrow C^*(E) \longrightarrow C(S^1) \longrightarrow 0, $$
see \cite{HaMaSz,HoSz3,HoSz}. It follows that $\Kk\oplus\Kk$ is the commutator ideal of $C^*(E)$ and 
thus it is invariant under all automorphisms of $C^*(E)$. In particular, the ideal $\Kk\oplus\Kk$ is 
invariant under the gauge action of the circle group on $C^*(E)$. Thus, by 
Theorem 3.6 and Corollary 3.5 of \cite{BHRS}, there exists a (non-empty) saturated hereditary 
subset $H$ of $E^0$ and a subset $B$ of $\hf$ such that $\Kk\oplus\Kk$ equals $J_{H,B}$ and 
$C(S^1)\cong C^*((E/H)\setminus\beta(B))$. (See the definitions of the relevant symbols in \cite{BHRS}.) 
It is not difficult to verify that there exists precisely one graph (comprised of one vertex and one edge) whose 
corresponding $C^*$-algebra is $C(S^1)$. It follows that $B=\hf$ and there is a vertex $w_0$ such that $E^0\setminus H=\{w_0\}$. Furthermore, there is exactly one edge $f$ with both the source 
$s(f)$ and the range $r(f)$ equal to $w_0$. There are no edges from $H$ to $w_0$ but there may exist 
edges from $w_0$ to $H$. 

Since the ideal $\Kk\oplus\Kk$ is gauge-invariant, it is itself isomorphic to a graph algebra, 
\cite[Lemma 1.6]{DHS}, and  the underlying graph contains the restriction of $E$ to $H$ as a subgraph. 
As $\Kk\oplus\Kk$ is an AF algebra, it follows that $E$ does not contain any loops other than $(f)$ (see 
\cite[Theorem 2.4]{KPR} and \cite[Remark 5.4]{RS}).  Consequently, $H$ must contain some sinks. Since 
each sink gives rise to an ideal isomorphic to the compacts (on some Hilbert space) 
and $C^*(E)$ contains precisely two such ideals, it follows that $H$ contains exactly two sinks, 
say $w_1,w_2$. Now, it is a consequence of the formula for $K$-theory of a graph algebra 
(for example, see  \cite[Theorem 6.1]{BHRS}, \cite[Theorem 3.1]{DT} or \cite[Proposition 2]{Sz}) 
that $E$ cannot contain vertices emitting infinitely many edges. Indeed, in such a case the 
$K_0$ group of $C^*(E)$ could not be isomorphic to $\ZZ\oplus\ZZ$, as the $K_0$ group 
of $C(S_{q,\beta}^2)$ (see \cite{HaMaSz} and \cite[Theorem 5.3]{HoSz}).  
Consequently, $E$ is a finite graph. 

Let $V$ be the smallest hereditary subset of $E^0$ containing $w_0$. Then both $w_1$ and $w_2$ belong 
to $V$. Indeed, otherwise $C^*(E)$ would admit a non-trivial splitting as a direct sum by 
\cite[Theorem 4.1]{H}. This however is not possible for $C(S^2_{q,\beta})$, in view of the 
description of its irreducible representations, \cite{HaMaSz}. We now define 
$$ T=\sum_{s(e)\in V}S_e \;\;\; \text{and} \;\;\; P=\sum_{v\in V}P_v.  $$
$T$ is an element of the corner $C^*$-subalgebra $PC^*(E)P$ and we have $T^*T\geq P$. Thus 
$T$ admits polar decomposition in $PC^*(E)P$, and hence also in $C^*(E)$: 
$$ T=U|T|. $$ 
By construction, $U$ is a partial isometry with domain projection $U^*U=P$. It is also clear that the 
range projection of $U$ is the sum of all those vertex projections $P_v$ such that $v\in V$ emits   
at least one edge. Thus we have 
$$ UU^*=P-(P_{w_1}+P_{w_2}). $$ 
Consequently, in the $K_0$ group of $C^*(E)$, we have 
$$ [P_{w_1}+P_{w_2}]=0, $$
where $P_{w_1}$ and $P_{w_2}$ are minimal projections in $\Kk\oplus 0$ and $0\oplus\Kk$, 
respectively. However, such relation does not hold in the $K_0$ group of $C(S^2_{q,\beta})$ 
(see  \cite{HaMaSz} and \cite[Theorem 5.3]{HoSz}). This is a contradiction with 
$C^*(E) \cong C(S^2_{q,\beta})$, and the proof is complete. 
\end{proof}

\vspace{6mm}\noindent
David Robertson\\
Department of Mathematics and Computer Science \\
The University of Southern Denmark \\
Campusvej 55, DK-5230 Odense M, Denmark \\
E-mail: dir@imada.sdu.dk

\vspace{5mm}\noindent
Wojciech Szyma{\'n}ski\\
Department of Mathematics and Computer Science \\
The University of Southern Denmark \\
Campusvej 55, DK-5230 Odense M, Denmark \\
E-mail: szymanski@imada.sdu.dk

\end{document}